\documentclass[12pt]{article}

\usepackage{amsmath}
\usepackage{amssymb}

\newcommand{\ifff}{ \ if and only if \ \ \ {}}
\newtheorem{theorem}{Theorem}

\newcommand{\imp}{\Rightarrow}

\newcommand{\conj}{\wedge}
\newcommand{\Conj}{\bigwedge}

\newcommand{\cart}{\times}    

\newcommand{\henk}[4]{
 {\left.
 \begin{array}{ll}
 \mbox{{$\forall #1\ \exists #2$}} \\
 \mbox{{$\forall #3\ \exists #4$}}
 \end{array}
 \right.} }

\newcommand{\vp}{\varphi}

\newcommand{\ntj}[2]{{{#1}_1,\ldots,{#1}_{#2}}}
\newcommand{\mc}[1]{\mathcal{#1}}
\newcommand{\Henk}{{\sf H}}

\newenvironment{proof}{{\bf Proof.}}{$\Box$\\}

\newcommand{\ulongmapsto}[1]{\underset{#1}{\longmapsto}}

\title{One Henkin Quantifier  in the empty vocabulary suffices for undecidability}

\author{ Konrad Zdanowski,\\
Faculty of Mathematics and Natural Sciences,  School of Exact Sciences,\\
Cardinal Stefan Wyszy{\'n}ski University in Warsaw }

\begin{document}
\maketitle
\begin{abstract}
We prove that there are single Henkin 
quantifiers such that first order logic
augmented by one of these quantifiers is undecidable
in the empty vocabulary. Examples of such quantifiers
are given. 
\end{abstract}
\section{Introduction}

In first order logic an existential variable  $y$ depends on all
universal variables $x$ such that $y$ lies in the scope of $x$. 
It follows that we can not express that in a predicate $P(x,y,z,w)$
a variable $y$ depends only on $x$ and $w$ depends only on $z$. 
To overcome this restriction Henkin proposed to use quantifiers
prefixes in which the ordering of 
variables is ony partial, not linear. Then, we could express 
dependences as above with  the following prefix:
$$
\henk{x}{y}{z}{w} P(x,y,z,w).
$$
Henkin, or branched, quantiers are a way of introducing dependences between variables
which are not expressible in first order logic. They occurred to be 
an interesting extension of first order logic which do not introduce
the full power of second order quantification.  Henkin quantifiers were examined 
in various contexts. Jaako Hintikka consider the following sentence  
of natural language:
\begin{center}
``{\em Some relative of each villager and some relative of each townsman hate each other.}''
\end{center}
His claim, known as Hintikka's Thesis, states that the logical form of the sentences as above
essentially requires branched quantification. We refer to Gierasimczuk and Szymanik \cite{GS} for a recent
discussion of Hintikka's Thesis. In complexity theory branched quantifiers
were examined as a way of capturing complexity classes by logics, see 
Blass and Gurevich \cite{BG} and Ko{\l}odziejczyk \cite{K}.

In this paper we prove that there are single Henkin quantifiers 
$H$ which give undecidable extenstion of first order logic already 
in the empty voca\-bulary. Previous results by Krynicki and Mostowski and 
by Mostowski and Zdanowski showed this property only for infinite classes 
of Henkin quantifiers.
\bigskip

\section{Basic notions}
We investigate different logics with Henkin quantifiers. The simplest
Henkin quantifier has the form 
$$
\henk{x}{y}{z}{w}.
$$
Intuitively, it expresses that the choice of $y$ does not 
depend on the variable $z$ and the choice of $w$ does not 
depend on $x$. 
More formally we can describe 
the Henkin prefix as an ordered triple
 $Q=(A,E,D)$, where $A$ and $E$ are disjoint sets of universal and
existential variables, respectively, and $D\subseteq A\cart E$ is
a dependency relation. 
We say that a variable $y\in E$ depends on a variable $x\in A$ if $(x,y)\in D$.
Further on, we will make no differences between 
quantifiers and quantifier prefixes.

\noindent{\bf Example.}\\ $\henk{x}{y}{z}{w}=
(\{x,z\},\{y,w\},\{(x,y),(z,w)\})$. 

\noindent
We denote the above quantifier by $\Henk$.

The inductive step in the  definition of semantics for
logic with Henkin quantifiers is as follows. Let 
$Q=(\{\ntj{x}{n}\},\{\ntj{y}{k}\},D)$. Then,
$$M\models Q\varphi(\ntj{x}{n},\ntj{y}{k})$$
$$\textrm{\ifff}$$
$$\textrm{there are operations } \ntj{f}{k} \textrm{ on } M 
\textrm{ such that }$$
$$(M,\{f_i\}_{i\leq k})\!\models\!\forall {\bf \overline{x}}
 \varphi({\bf \overline{x}},f_1(\overline{x}_1),
 \ldots,f_k(\overline{x}_k)),$$
 where ${\bf \overline{x}}$ are all universal variables
 in $Q$ and $\overline{x}_i$ are variables on which $y_i$ depends in $Q$.

By $\mc{H}$ we denote the family of all Henkin quantifiers.
For  a family of Henkin quantifiers $\mc{Q}$, $L(\mc{Q})$
is an extensions of the first order logic 
by quantifiers  in $\mc{Q}$. For a single quantifier $Q$ we
write $L(Q)$ for $L(\{Q\})$.

The logic with Henkin quantifiers was shown to be a strengthening
of first  order logic by Ehrenfeucht. He showed  that one can define the
finitness of the universe by the following sentence.
$$
\neg\exists t\henk{x}{y}{z}{w}(y=w\equiv x=z)\conj(t\not=y).
$$
The sentence above is equivalent to the second order sentence 
$$
\neg\exists t\exists f\forall x, y (f(x)=f(y)\imp x=y)\conj
\forall x(t\not=f(x))
$$ 
which states that there is no injection
of the universe of a given model into itself which is not 
a bijection.  

We have the following theorem relating the semantical
power of logic with Henkin quantifiers with that of second
order logic. The first dependence was independently proved
by Enderton and Walkoe, the second is due to Enderton.
 \begin{theorem}[see \cite{E}, \cite{W}]
$\Sigma^1_1 \leq L(\mc{H}) \leq \Delta^1_2$, where $\mc{H}$ is the
family of all Henkin quantifiers.
 \end{theorem}
It should be added that all the inequalities above are strict.
The first one is obvious since $L({\mc{H}})$ is closed
on the negation and $\Sigma^1_1$ is not. The second one
was proven by M.~Mostowski in \cite{Mb} by means of 
truth definitions. For a simpler
argument which works for the empty vocabulary see \cite{MZb}.


We will consider the following kinds of Henkin quantifiers.
By ${\sf H}_n x_1\ldots x_n\ y_1\ldots y_n$  we denote 
the quantifier
$$
\begin{array}{c}
\forall x_1\ \exists y_1\\
\forall x_2\ \exists y_2\\ 
\ldots\ldots\ldots \\ \forall x_n\
\exists y_n\\
\end{array}
$$

By ${\sf E}_n x_1\ x_2\ y_1\ldots y_n\ z_1\ldots z_n$ we denote
$$
\begin{array}{c}
\forall x_1\ \exists y_1\ldots y_n\\ \forall x_2\ \exists
z_1\ldots z_n\\
\end{array}
$$

By ${\sf H}_\omega$ we denote the family of Henkin quantifiers
$\{{\sf H}_n\}_{n=2,3,\ldots}$ and similarly for ${\sf E}_\omega$.

Clearly, each quantifier ${\sf E}_n$ can be defined in the
logic with quantifier ${\sf H}_n$. However, it is not known if for each
$k$ there is $n$ such that $L({\sf H}_k)\leq L({\sf E}_n)$.

Now, we present known results
on decidability of different logics with
Henkin quantifiers. Our aim is to outline 
for these logics the
boundary between decidable and undecidable.

\begin{theorem}[\cite{KL}]
Let $\sigma$ be a monadic vocabulary. Logic $L_\sigma({\sf H}_2)$
is decidable.
\end{theorem}
\begin{theorem}[\cite{KL}]\label{Hun}
Let $\sigma$ contains one unary function symbol. Then 
logic $L_\sigma({\sf H}_2)$ is undecidable.
\end{theorem}
\begin{theorem}[\cite{KM}]
Let $\sigma$ be an infinite monadic vocabulary.
Then $L_\sigma({\sf H}_4)$ is undecidable.
\end{theorem}

The proof of theorem \ref{Hun} gives an up-to-isomorphism 
a characterization  
of the standard model of arithmetic in the language of $L_\sigma({\sf H}_2)$. 
An unary function symbol is intended
there to be a successor function. Similarly, definitions
of addition and multiplication by means of a successor
function are given. In \cite{MZa} it was observed that also for some finite
monadic vocabulary $\tau$ one obtain undecidable logic
 $L_\tau({\sf H}_4)$.

As far as the  empty vocabulary is concern it was not known
whether there exists a single Henkin quantifier $Q$ such that 
$L_\emptyset(Q)$ is undecidable.
The only undecidability results were established for 
the infinite families ${\sf H}_\omega$ (\cite{KM}) and ${\sf E}_\omega$ (\cite{MZa}).
\begin{theorem}[\cite{KM},\cite{MZa}]\label{Homega}
Logics $L_\emptyset({\sf H}_\omega)$ and $L_\emptyset({\sf
E}_\omega)$ are undecidable.
\end{theorem}

In the next section we prove that there is one 
Henkin quantifier for which we obtain undecidable
logic in the empty vocabulary. We present also 
examples of such quantifiers.

\section{Undecidable logics with one Henkin quantifier}

Firstly, we prove that there is a single Henkin quantifier 
such that the logic with this quantifier is undecidable 
in the empty vocabulary.
Next, we give an estimation of a size of such quantifier.
Our proof is a modification of proofs of Theorem \ref{Homega} 
as presented in \cite{KM} and \cite{MZa}. 
Krynicki and Mostowski  gave in \cite{KM} a reduction of the word problem
for semigroups to the tautology problem for $L_\emptyset({\sf H}_\omega)$. 
We carry out this method  
in a way which allows us to obtain a single Henkin  quantifier 
${\sf H}_n$ or ${\sf E}_n$ such that the logic 
with this quantifier is undecidable in the empty vocabulary.

\begin{theorem}\label{existsn}
There is $n$ such that logics 
$L_\emptyset({\sf H}_{n})$ and $L_\emptyset({\sf E}_{n})$
are undecidable.
\end{theorem}
\begin{proof}
Let $\Sigma=\{a,b\}$ be an alphabet and let $E=\{v_i=w_i:i\leq m
\conj w_i,v_i\in\Sigma^*\}$ be a semigroup. The word problem for $E$
is the set of equations $v=w$ of words from $\Sigma^*$ such that 
any semigroup satisfying $E$ satisfies also $v=w$. We denote this by
$E\models v=w$. 
Let us fix such a semigroup $E$ that its word problem is undecidable.

For each letter $x$ in $\Sigma$ we fix a function symbol  $f_x$ 
and by $f\circ g$ we denote the composition of $f$ and $g$.
For a word  ${c_1}\ldots{c_k}\in\Sigma^*$ we define
the translation $tr$ as follows,
$tr({c_1}\ldots{c_k})=f_{c_1}\circ\ldots\circ f_{c_k}$.

By the representation theorem for semigroups
each semigroup is isomorphic to a semigroup of unary functions
with the composition as the semigroup operation. Thus we
have that 

$$\begin{array}{l}E\not\models v=w \textrm{\ifff}\\
  \exists M \exists f_{a}\ f_{b}
   \textrm{ unary operations on $M$ such that}\\
   (M,f_{a},f_b)
   \models\Conj_{i\leq m}\forall x\ tr(v_i)(x)=tr(w_i)(x)
  \conj \exists x\ tr(v)(x)\not= tr(w)(x).
\end{array}$$

Let $v=v_1\ldots v_m$ and $w=w_1\ldots w_k$ be arbitrary
words over $\Sigma$.
 Then we can express
 $\exists f_a\exists f_b\forall x (tr(v)(x)=tr(w)(x))$ by means of
some Henkin quantifier ${\sf H}_n$ and the following formula
 $$
\begin{array}{ll}
 \forall x_{1}& \exists y_{1}\\
 \ldots&\ldots\\
 \forall z_{m}& \exists y_{m}\\
 \forall z_{1}& \exists r_{1}\\
 \ldots&\ldots\\
 \forall z_{k}& \exists r_{k}\\
\end{array} (\varphi_0 \conj \varphi_{v=w}),
$$
where
\begin{align*} 
\varphi_0\  = &\ \  \Conj_{v_i=v_j}(x_i=x_j\imp y_i=y_j)
   \conj \Conj_{w_i=w_j}(z_i=z_j\imp r_i=r_j)\conj\\
 \  &\ \ \Conj_{v_i=w_j}(x_i=z_j\imp y_i=r_j),\\
   \varphi_{v=w}\ =&\ \ ((\Conj_{1\leq i<m}(x_{i}=y_{i+1})
   \conj\Conj_{1\leq i<k}(z_{i}=r_{i+1}))
   \imp (x_m=z_k\imp y_1=r_1).
\end{align*}

The formula $\varphi_0$ says that  
the choice functions are the same if their rows  
represent the same letter. The formula $\varphi_{v=w}$ expresses
the fact that if the values of $x$'s and $z$'s satisfy the dependences
of the diagram below and $x_m=z_k$, then $f_{v_1}(x_1)=f_{w_1}(z_1)$.
We may depict it as follows. An arrow of the form
$y \ulongmapsto{f} z$ indicates
that $z = f(y)$. Thus, the predecessor of $\varphi_{v=w}$ expresses 
the following dependences:
\begin{gather*}
x_m \ulongmapsto{f_{v_m}} x_{m-1}\ulongmapsto{f_{v_{m-1}}}
\ldots\ulongmapsto{f_{v_2}}x_1\ulongmapsto{f_{v_1}}y_1, \\
z_k\ulongmapsto{f_{w_k}}z_{k-1}\ulongmapsto{f_{w_{k-1}}}\ldots
\ulongmapsto{f_{w_2}}z_1\ulongmapsto{f_{w_1}}r_1
\end{gather*}
Then, equality $y_1=r_1$ means that $tr(v)(x_m)=tr(w)(z_k)$. Since $x_m$ and $z_k$ are quantified
universally and we assume their equality
this is equivalent to  $\forall x (tr(v)(x)=tr(w)(x))$.

Next, we choose $n$ big enough to express $\exists f_a\ f_b
(\Conj_{i\leq m} \varphi_{v_i=w_i})$ in $L({\sf H}_n)$. Now, we need 
to observe that in order to express
$\exists x\ tr(v)(x)\not=tr(w)(x)$ it suffices to add only first order
quantification, no matter how long are words $v$ and $w$. This is the place
when we modify previous constructions in order to stay with a fixed Henkin quantifier. 
To show  this let us assume that 
the choice functions for $y$
and  $r$ below are respectively  $f_a$, $f_b$ and that  
$v=v_1\ldots v_l$ and
$w=w_1\ldots w_k$.

Let us consider the following formula,
\begin{equation}\label{MainEq}
\exists t_0\ldots t_l\ s_0\ldots s_k
 {\begin{array}{ll}
 \forall x&\!\!\!\! \exists y\\
 \forall z&\!\!\!\! \exists r\\
 \forall x_3&\!\!\!\! \exists y_3\\
 \ldots&\ldots\\
 \forall x_{n}&\!\!\!\! \exists y_{n}\\
\end{array}
} (\varphi_0\conj (\Conj_{i\leq m} \varphi_{v_i=w_i}) \conj
\varphi_{v\not=w}),\end{equation}
where
\begin{align*}
\varphi_{v\not=w}\ =&\ \ \Conj_{v_i=a}(x=t_{i}\imp y=t_{i-1})
\conj\Conj_{v_i=b}(z=t_{i}\imp r=t_{i-1})\conj\\
\ &\ \ \Conj_{w_i=a}(x=s_{i}\imp y=s_{i-1})
\conj\Conj_{w_i=b}(z=s_{i}\imp r=s_{i-1})
\conj\\
\ &\ \ (t_l=s_k)\conj (t_0\not=s_0).
\end{align*}

Here, $\varphi_{v\not=w}$ states that we can find 
in a given semigroup two 
sequences of elements, $t_l, \ldots, t_0$ and 
$s_k,\ldots, s_0$ such
 that the values of terms 
$tr(v)$ and $tr(w)$ on the $t_l$ and $s_k$
 are different. But since $t_l=s_k$, it follows that 
$\exists x\ tr(v)(x)\not=tr(w)(x)$.

Below we present the dependencies which satisfy the elements
of these two sequences as it is described by 
$\varphi_{v\not= w}$.
\begin{gather*}
t_m \ulongmapsto{f_{v_m}}
t_{m-1}\ulongmapsto{f_{v_{m-1}}}\ldots\ulongmapsto{f_{v_1}}t_0,\\
s_k \ulongmapsto{f_{w_k}}s_{k-1}\ulongmapsto{f_{w_{k-1}}}\ldots\ulongmapsto{f_{w_1}}s_0.
\end{gather*}
It follows that the formula \eqref{MainEq} is satisfiable
if and only if there is a semigroup $M$ with generators 
$a, b$ such that it satisfies all equations 
from $E$ and $M\models v\not =w$. 
Therefore, we reduced the problem whether $E\not\models v=w$
to the satisfability problem for $L({\sf H}_n)$. 
It should be noted that a similar construction works also
in a case of sufficiently large quantifier ${\sf E}_n$. See \cite{MZa} 
and below where we construct explicit formulas describing
the equations from a given  semigroup in the logic $L({\sf E}_n)$.
\end{proof} 

\section{An estimation of a size of quantifiers $H$ with undecidable logic $L_\emptyset(H)$}

Now, we give an estimation of the value of $n$ for which we get 
undecidable logics $L({\sf H}_n)$ and $L({\sf E}_n)$. 
Let $C$ be the semigroup with generators $a,b,c,d,e$,
defined by the following equations: 
$$
ac=ca, ad=da, bc=cb, bd=db, eca=ce, edb=de, cca=ccae.
$$
Ceitin proved that the word problem the semigroup $C$ 
is undecidable, see \cite{C} or chapter A.4 of \cite{MP}.
\begin{theorem}[Ceitin]
The word problem for $C$ is undecidable.
\end{theorem}

Having fixed a single semigroup with undecidable word 
problem we can explicitly construct a quantifier.
Below we describe the formulas with quantifiers 
${\sf H}_{12}$ and ${\sf E}_{10}$ 
which express that the functions $f_a,\ldots, f_e$ satisfy
the equations from the Ceitin's semigroup.
It follows that

\begin{theorem}
The logics $L_\emptyset({\sf H}_{12})$ and 
$L_\emptyset({\sf E}_{10})$ are undecidable.
\end{theorem}
\begin{proof}
The following formula describes the equations from the 
semigroup $C$.

$$
\begin{array}{l}
 \forall x_a\ \exists y_a\\
 \forall x'_a\ \exists y'_a\\
 \forall x_b\ \exists y_b\\
 \forall x'_b\ \exists y'_b\\
 \forall x_c\ \exists y_c\\
 \forall x'_c\ \exists y'_c\\
 \forall x_d\ \exists y_d\\
 \forall x'_d\ \exists y'_d\\
 \forall x_e\ \exists y_e\\
 \forall x'_e\ \exists y'_e\\
 \forall x_{cc}\ \exists y_{cc} \\
 \forall x'_{cc}\ \exists y'_{cc} \\
\end{array}{(\psi \conj \varphi \conj \Conj_{i< 7}\varphi_i}),
$$
where 
\begin{align*}
\psi &= \Conj_{q\in\{a,b,c,d,e,cc\}}(x_q=x'_q \imp y_q=y'_q),\\
\varphi &= (x_c=x_{cc} \conj y_c=x'_c \imp y'_c=y_{cc}),\\  
\vp_{0}&= (x_a=x_c\conj x'_a=y_c\conj x'_c=y_a\imp y'_c=y'_a),\\ 
\vp_1&=(x_a=x_d\conj x'_a=y_d\conj x'_d=y_a\imp y'_d=y'_a),\\ 
\vp_{2}&= (x_b=x_c\conj x'_b=y_c\conj x'_c=y_b\imp y'_c=y'_b),\\ 
\vp_3&=(x_b=x_d\conj x'_b=y_d\conj x'_d=y_b\imp y'_d=y'_b),\\ 
\vp_4&= (x_a=x'_e \conj y_a=x_c \conj y'_e= x'_c \conj x_e=y_c \imp y_e=y'_c),\\
\vp_5&=(x_b=x'_e \conj y_b=x_d \conj y_d= x_e \conj y'_e=x'_d \imp y_e=y'_d),\\
\vp_6&= (x_a=x'_e \conj y_a=x_{cc} \conj y'_e=x'_a \conj y'_a = x'_{cc} \imp y_{cc}=y'_{cc}).
\end{align*}
The formula $\psi$ expresses the fact that variables $y_q$ and $y'_q$ describe the 
same functional dependency, for $q\in\{a,b,c,d,e,cc\}$. The formula $\varphi$ expresses
that the choice function for $y_{cc}$ (and, implicitly, for $y'_{cc}$) 
is just a composition of a function for $y_c$ with itself.
The formulas $\vp_i$ describe the $i$-th equations from the semigroup
$C$ given above. It should be clear that indices of variables indicate 
what kind of function or a composition of functions they represent. 

Now let us describe the equations from $C$ with
the quantifier ${\sf E}_{10}$.
The formula has the form
$$\begin{array}{l}
\forall x_1\ \exists y_a\ \exists y_{ca}\ \exists y_{da}\  
 \exists y_{b}\ \exists y_{cb}\ 
 \exists y_{db}\ \exists y_{e}\ \exists y_{eca}\  
 \exists y_{de}\ \exists y_{cca}\\
 \forall x_2\ \exists y_c\ \exists y_{ac}\ \exists y_{d}\  
 \exists y_{ad}\ \exists y_{bc}\ 
 \exists y_{bd}\ \exists y'_{e}\ \exists y'_{cca}\\
\end{array}(\gamma\conj\gamma_{0123}\conj \Conj_{4\leq i<7}\gamma_i).
$$
Above, the formula $\gamma$ establishes that existential variables
describe the function compositions according to
their subscripts. It has the following form:\\ 
\begin{align*}
\gamma =& (y_a=x_2\imp y_c=y_{ca})\conj(y_c=x_1\imp y_a=y_{ac})\conj
(y_a=x_2\imp y_{da}=y_d)\conj \\ 
& (y_d=x_1\imp y_{ad}=y_a)\conj
(y_b=x_2\imp y_{cb}=y_c)\conj (y_c=x_1\imp y_b=y_{bc})\conj\\
& (y_b=x_2\imp y_{db}=y_d)\conj(y_d=x_1\imp y_{bd}=y_b)\conj (x_1=x_2\imp y_e=y'_e)\\
&(y_{ca}=x_2\imp y_{eca}=y'_{e})\conj (y_e=x_2\imp y_{de}=y_{d})\conj\\
&(y_{ca}=x_2\imp y_{cca}=y_{c})\conj(x_1=x_2\imp y_{cca}=y'_{cca}).\\
\end{align*}
The formulas $\gamma_x$, for $x\in\{0123, 4, 5, 6\}$ state that axioms of Ceitin's semigroup 
$C$ are true for these functions.
For brevity we grouped the first four equations into one axiom.
\begin{align*}
\gamma_{0123}&=(x_1=x_2\imp (y_{ca}=y_{ac}\conj y_{ad}=y_{da}\conj
y_{bc}=y_{cb}\conj y_{db}=y_{bd})),\\
\gamma_4&= (y_e=x_2\imp y_{eca}=y_{c}),\\
\gamma_5 &= (y_{db}=x_2\imp y_{de}=y'_e),\\
\gamma_6 &= (y_e=x_2\imp y_{cca}=y'_{cca}).
\end{align*}

Now, to express for arbitrary words $v,w$ over the alphabet
$\{a,\ldots,e\}$ that $C\not\models v=w$ it suffices to follow 
the proof of theorem \ref{existsn}. One need only to add a proper 
 first order prefix to formulas above and the
formula $\varphi_{v\not= w}$. Thus, we reduced the problem
whether $C\not\models v=w$ to the satisfability problem for
$L({\sf H}_{12})$ or $L({\sf E}_{10})$. 
\end{proof}




\section{Conclusions}

We showed that there are single, relatively simple, Henkin quantifiers
$H$ such that the first order logic augmented with $H$ is undecidable
already in the empty vocabulary.
However, there is a considerable gap
between the decidable logic $L_\emptyset({\sf H}_2)$ (see \cite{KL})
and undecidable logics $L_\emptyset({\sf H}_{12})$ 
and $L_\emptyset({\sf E}_{10})$. It would be desirable to close this gap or,
at least, make it smaller.

Moreover, we did not touch a question of decidability of these logics
in finite models. Articles by Gurevich \cite{G} and by Gurevich and Lewis \cite{GL}
could be a good starting point for investigating this problem in finite models.
However, if one aims at small quantifiers it may be better to construct by hand
a semigroup with the undecidable word problem in the class of finite semigroups.

Finally, let us mention 
that Mostowski and Zdanowski proved in \cite{MZa} 
 that logics  $L^k_\emptyset(Q)$, for all $k$ and $Q$, are decidable in 
the class of infinite models only. However, we also  know
that for sufficiently large $k$ and $Q$ no algorithm can 
be proved in $\textrm{ZFC}$ as deciding the tautology problem  
for the logic $L^k_\emptyset(Q)$ (see \cite{Ma}). Here again, the complexity
of logics $L^k_\emptyset(Q)$ in finite models is unknown.


\end{document}